\documentclass[12pt,reqno]{amsart}

\usepackage{amssymb,amsmath,amsthm,mathtools,wasysym,calc,verbatim,enumitem,tikz,pgfplots,hyperref,url,mathrsfs,cite,fullpage,bbm,comment}
\usepackage{comment}
\mathtoolsset{showonlyrefs}
\usepackage{scalerel}
\usepackage{stackengine}
\stackMath

\newcommand\pig[1]{\scalerel*[5pt]{\normalsize#1}{%
  \ensurestackMath{\addstackgap[0.5pt]{\normalsize#1}}}}
  \newcommand\pigl[1]{\hspace{0.05cm}\mathopen{\pig{#1}}\hspace{0.03cm}}

\newtheorem{theorem}{Theorem}
\newtheorem{definition}[theorem]{Definition}

\newtheorem{lemma}[theorem]{Lemma}
\newtheorem{claim}[theorem]{Claim}

\newtheorem{observation}[theorem]{Observation}

\newtheorem*{claim*}{Claim}

\theoremstyle{remark}
\newtheorem*{remark*}{Remark}

\numberwithin{theorem}{section}

\renewcommand{\phi}{\varphi}

\newcommand{\gap}{\hspace{0.02cm}}

\newcommand{\p}{\hspace{0.02cm}+\hspace{0.02cm}}
\newcommand{\m}{\hspace{0.02cm}-\hspace{0.02cm}}

\newcommand{\ina}{\hspace{0.02cm}\in\hspace{0.02cm}}

\newcommand{\lea}{\hspace{0.03cm} \le \hspace{0.04cm}}
\newcommand{\gea}{\hspace{0.035cm} \ge \hspace{0.03cm}}
\newcommand{\eqa}{\hspace{0.02cm}=\hspace{0.015cm}}

\renewcommand{\leq}{\le}

\newcommand{\eps}{\varepsilon}

\newcommand{\cF}{\mathcal F}
\newcommand{\cE}{\mathcal E}

\newcommand{\cA}{\mathcal A}

\newcommand{\cH}{\mathcal H}

\newcommand{\cT}{\mathcal T}

\newcommand{\cJ}{\mathcal J}

\def\1{\mathbbm{1}}

\renewcommand{\le}{\leqslant}
\renewcommand{\ge}{\geqslant}

\newcommand{\Ex}{\mathbb E}
\renewcommand{\Pr}{\mathbb{P}}

\newcommand{\N}{\mathbb N}

\def\va{\mathbf{a}}
\def\vb{\mathbf{b}}

\begin{document}

\title{On the Erd\H{o}s--Rogers function} 

\author{Robert Morris \and Julian Sahasrabudhe \and Jacques Verstra\"ete}

\address{IMPA, Estrada Dona Castorina 110, Jardim Bot\^anico, Rio de Janeiro, 22460-320, Brasil}\email{rob@impa.br}

\address{Department of Pure Mathematics and Mathematical Statistics, Wilberforce Road, Cambridge, CB3 0WA, UK}
\email{jdrs2@cam.ac.uk}

\address{Department of Mathematics, University of California, San Diego, La Jolla, CA 92093}
\email{jacques@ucsd.edu}

\thanks{RM was partially supported by FAPERJ (Proc.~E-26/200.977/2021) and by CNPq (Procs~303681/2020-9 and~407970/2023-1); JS was partially supported by European Research Council (ERC) Starting Grant “High Dimensional Probability and Combinatorics”, grant No.~101165900; JV was partially supported by the National Science Foundation FRG Award DMS-1952786 and NSF Award DMS-2347832.}

\begin{abstract}
We show that the Erd\H{o}s--Rogers function $f_{s,s+1}(n)$ satisfies 
$$f_{s,s+1}(n) = \Theta\big( \sqrt{n \log n} \, \big)$$
for every $s \ge 2$. More precisely, we construct a $K_{s+1}$-free graph on $n$ vertices in which every set of at least $C(s) \displaystyle{ \sqrt{n \log n}}$ vertices contains a copy of $K_s$ for some constant $C(s)$, which implies the upper bound. 
The matching lower bound follows from a theorem of Joret, Micek, Reed and Smid on the clique chromatic number of a graph.
\end{abstract}

\maketitle

\vspace{-2em}

\section{Introduction}

The Ramsey number $R(\ell,k)$ is the smallest $n \in \N$ such that every red-blue colouring of the edges of the complete graph $K_n$ contains either a red copy of $K_\ell$ or a blue copy of $K_k$. The study of these numbers was introduced by Ramsey~\cite{R30} and Erd\H{o}s and Szekeres~\cite{ESz} in the 1930s, and has been a central topic in probabilistic combinatorics for well over half a century, since the seminal papers of Erd\H{o}s~\cite{E47,E59,E61}. In recent years, the area has seen a number of exciting advances (see, for example,~\cite{BBCGHMST,Br26,CGMS,CJMS25,HHKP25,MaV}); we refer the reader to the surveys~\cite{ICM26,Ver26} for a gentle introduction to these (and other) recent developments. 

When $\ell$ is fixed and $k \to \infty$, the problem of proving bounds on $R(\ell,k)$ is asymptotically equivalent to that of bounding the size of the largest independent set that can be found in every $K_\ell$-free graph with $n$~vertices. The following natural generalisation of this problem was introduced in 1962 by Erd\H{o}s and Rogers \cite{ER62}, who were motivated by a question of Hajnal: What is the largest $K_s$-free subset that can be found in every $K_t$-free graph on $n$ vertices? 

In this paper, we will consider the case $t = s + 1$ of this problem, which has received a significant amount of attention over the past 15 years. To state the known bounds for this problem, let us write $f_s(n)$ for the size of the largest $K_{s}$-free set of vertices that can be found in every $K_{s+1}$-free graph with $n$ vertices, and observe that
$$f_2(n) \ge k \qquad \; \Leftrightarrow \; \qquad n \ge R(3,k).$$
By the famous results of Ajtai, Koml\'os and Szemer\'{e}di~\cite{AKSz80} and Kim~\cite{Kim95}, which determined $R(3,k)$ up to a constant factor, it follows that
$$f_2(n) = \Theta\big( \sqrt{n \log n} \, \big).$$
In fact, much more is now known: the lower bound was improved by Shearer~\cite{S83} in 1983, and the upper bound was improved in a series of papers~\cite{Boh,BK,FGM,CJMS25}, culminating in the work of Hefty, Horn, King and Pfender~\cite{HHKP25}, who showed that
\begin{equation}\label{eq:R3k:bounds}
\bigg( \frac{1}{\sqrt{2}} + o(1) \bigg) \sqrt{n \log n} \le f_2(n) \le \big( 1 + o(1) \big) \sqrt{n \log n}
\end{equation}
as $n \to \infty$. It seems likely that any further improvement of the constant factor in either of these bounds would require a major breakthrough. 

When $s \ge 3$, the function $f_s(n)$ is no longer equivalent to the Ramsey numbers, and progress on the problem has been somewhat slower. The lower bound 
\begin{equation}\label{eq:JMRS:lower}
f_s(n) \ge c \sqrt{n \log n}
\end{equation}
for some absolute constant $c  > 0$ follows easily from a theorem of Joret, Micek, Reed and Smid~\cite{JMRS21} on the clique chromatic number of a graph\footnote{We are grateful to Huy Pham for pointing out this theorem to us.} (see below), the proof of which was based on an earlier approach of Molloy~\cite{Mol19}. For comparison, the bound $f_s(n) \ge \sqrt{n}$ (first observed in~\cite{BH91}) is straightforward,\footnote{To see this, simply note that the neighbourhood of any vertex in a $K_{s+1}$-free graph $G$ is $K_s$-free, and recall that $\alpha(G) \ge n / (\Delta(G) + 1)$. Dudek and Mubayi~\cite{DM14} were the first to observe that if we replace this by Shearer's bound $\alpha(G) \ge c(s) (n \log d) / (d \log \log d)$, which holds for $K_{s+1}$-free graphs with $n$ vertices and maximum degree $d$, then we obtain a bound of the form $f_s(n) = \Omega_s\big( \displaystyle{\sqrt{ n \log n / \log\log n}} \big)$.} and a well-known theorem of Shearer~\cite{S95}, which bounds the independence number of a $K_{s+1}$-free graph $G$ with maximum degree $d$, gives a bound that is within a factor of order $(\log\log n)^{1/2}$ of~\eqref{eq:JMRS:lower} for fixed $s$.

The first non-trivial upper bound on $f_s(n)$ in the case $s \ge 3$ was obtained by Erd\H{o}s and Rogers~\cite{ER62}, using a random geometric graph, and was not improved until almost 30 years later, when Bollobás and Hind~\cite{BH91} introduced the important idea of constructing a graph from a random collection of copies of $K_s$, and showed that $f_3(n) \le n^{7/10 \p o(1)}$. This was later strengthened to $f_3(n) = O\big( n^{3/5} (\log n)^{1/2} \big)$ by Krivelevich~\cite{Kr94,Kr95}, but for large $s$ the methods of the papers~\cite{ER62,BH91,Kr94,Kr95} only imply bounds of the form $f_s(n) \le n^{1 \m \eps(s)}$ for some $\eps(s)$ tending to zero as $s \to \infty$, and for many years it was unclear whether or not there should exist a constant $\eps > 0$ such that $f_s(n) \le n^{1 \m \eps}$ for all $s \in \N$ (see~\cite{Kr95,Su05b}). 


A key breakthrough was made in 2011 by Dudek and R\"odl~\cite{DR11}, who resolved this question by proving an upper bound of the form
$$f_s(n) = O(n^{2/3})$$
for every $s \ge 3$. To do so, they constructed a graph by placing random blow-ups of $K_s$ in the sets of an algebraic object called the `generalised quadrangle'. Their method was then developed further by Wolfovitz~\cite{W13} and Dudek, Retter and R\"odl~\cite{DRR14}, who proved that 
\begin{equation}\label{eq:W:DRR}
f_{s}(n) = \sqrt{n} \pigl( \log n \pig)^{O(s^2)}
\end{equation}
for every $s \ge 3$. To do so, they replaced the generalised quadrangle used in~\cite{DR11} by a random subset of the lines of a projective\footnote{More precisely, Wolfovitz used a projective plane to prove~\eqref{eq:W:DRR} in the case $s = 3$, while Dudek, Retter and R\"odl used an affine plane (and a more involved edge-deletion step) to extend his result to $s \ge 4$.} plane, and also added a final (random) edge-removal step. More recently, Mubayi and Verstra\"ete~\cite{MV25} improved the bound further, showing that  
$$f_{s}(n) \le C(s) \sqrt{n} \log n$$
for every $s \ge 3$ and some constant $C(s) > 0$. To do so, they replaced the hypergraph used in~\cite{DRR14} by a random sub-hypergraph of the one used in the recent breakthrough result of Mattheus and Verstra\"ete~\cite{MaV} on the Ramsey numbers $R(4,k)$, which was constructed using an algebraic object called the Hermitian unital. They also tentatively conjectured (see~\cite[Conjecture~1]{MV25}) that moreover $f_{s}(n) = \sqrt{n} \pigl( \log n \pig)^{1-o(1)}$ for all $s \ge 3$. 

In this paper we disprove this conjecture by showing that the correct exponent of $\log n$ is in fact $1/2$, and moreover determine $f_s(n)$ up to a constant factor. 

\begin{theorem}\label{thm:ER} 
For every $s \ge 2$, we have 
\begin{equation}\label{eq:mainthm}
f_s(n) = \Theta\big( \sqrt{n \log n} \, \big) 
\end{equation}
as $n \to \infty$. 
\end{theorem}

An interesting aspect of our construction is that, unlike those in~\cite{DR11,W13,DRR14,MV25}, it makes no use of algebraic objects such as the generalised quadrangle, the projective plane, or the Hermitian unital.
Indeed, the first step towards our construction was the observation that simply throwing down random blow-ups of $K_s$, 
and deleting edges in an appropriate way (as in the earlier construction of Gowers and Janzer~\cite{GJ20}, see Section~\ref{sec:cliques}), does just as well as (and, in fact, a little better than) the earlier algebraic approaches. However, it does not seem possible to reduce the power of $\log n$ below $3/4$ with this approach alone. 

The second key idea that we use is to combine two random blow-ups of this graph, as in the recent breakthrough of Hefty, Horn, King and Pfender~\cite{HHKP25} on the Ramsey numbers $R(3,k)$ (see also the more classical work of Alon and R\"odl~\cite{AR} for multicolour Ramsey numbers). At first glance it does not seem possible to use this idea here, as the method is limited to situations where we are working with triangle-free graphs (or rather with odd-cycle-free graphs, as in~\cite{CJMPS25}). However, we are \emph{almost} building triangle-free graphs, in the sense that the only triangles in our constructions are contained in the random blowups of $K_s$. 

The constant we obtain in the upper bound of Theorem~\ref{thm:ER} is a polynomial function of~$s$ (see Section~\ref{sec:triangle-free} for a discussion of the limits of our method in this direction). On the other hand, the constant in the lower bound does not depend on $s$, since by~\cite[Corollary~2]{JMRS21}, every graph with $n$ vertices has a vertex-colouring with $O\big( \sqrt{n / \log n} \, \big)$ colours such that no maximal clique is monochromatic, and if $G$ is $K_{s+1}$-free, then the largest colour class in such a colouring is a $K_s$-free subset of $V(G)$ of size $\Omega\big( \displaystyle{\sqrt{n \log n}} \, \big)$. It would be extremely interesting to determine the behaviour of $f_s(n)$ as $s \to \infty$.

\subsection{The Erd\H{o}s--Rogers function for pairs of cliques}\label{sec:cliques}

Let us write $f_{s,t}(n)$ for the size of the largest $K_s$-free subset that can be found in every $K_t$-free graph on $n$ vertices, and note in particular that 
$$f_{s,s+1}(n) = f_s(n) \quad \qquad \text{and} \qquad \quad f_{2,t}(n) \ge k \quad  \Leftrightarrow \quad n \ge R(t,k).$$
Erd\H{o}s and Rogers~\cite{ER62} also introduced this more general problem, but the first bounds in the case $t \ge s + 2 \ge 5$ were obtained by Bollobás and Hind~\cite{BH91} and by Krivelevich~\cite{Kr94,Kr95}, who showed that
\begin{equation}\label{eq:fst:general:bounds}
n^{1/(t-s+1) \p o(1)} \le f_{s,t}(n) \le n^{s/(t+1) \p o(1)}
\end{equation}
as $n \to \infty$. The lower bound was improved by Sudakov~\cite{Su05a,Su05b} using dependent random choice (see the survey~\cite{DRC} for an introduction to this beautiful technique), who showed (amongst other results) that
$$f_{3,5}(n) \ge n^{5/12} \qquad \text{and} \qquad f_{s,s+2}(n) \ge n^{1/2 \m c(s)}$$
for some function $c(s) = \Theta(1/s)$, for every $s \ge 3$ and all sufficiently large $n \in \N$. Despite the significant amount of interest in this problem over the past 20 years, there has been no subsequent progress on the lower bounds when $t \ge s + 2$.

An upper bound of the form $f_{s,s+2}(n) = O(n^{1/2})$ was obtained by Dudek, Retter and R\"odl~\cite{DRR14} for all $s \ge 4$, and this was improved by Gowers and Janzer~\cite{GJ20} to 
$$f_{3,5}(n) \le n^{6/13 \p o(1)} \qquad \text{and} \qquad f_{s,s+2}(n) \le n^{1/2 \m c'(s)}$$
for some function $c'(s) = \Theta(1/s)$, for every $s \ge 3$. The first step of our construction is very similar to theirs, which also uses a random collection of blow-ups of $K_s$. More recently, Janzer and Sudakov~\cite{JS22} adapted the approach of Mattheus and Verstra\"ete~\cite{MaV}, who used the Hermitian unital to determine $f_{2,4}(n)$ up to a poly-logarithmic factor of $n$, to show that
$$f_{s,s+2}(n) = O\big( n^{(2s-3)/(4s-5)} \pigl( \log n \pig)^3 \big)$$
for every $s \ge 3$, which in particular implies that $f_{3,5}(n) \le n^{3/7 \p o(1)}$. Even more recently, Brada\v{c}~\cite{Br26} dramatically improved the best known lower bound on $R(s,k)$ for all $s \ge 5$, showing that $R(s,k) = k^{s - 1 + o(1)}$. As noted in~\cite[Section~4]{Br26}, this implies the bound 
$$f_{s,t}(n) \le n^{(s-1)/(t-1) \p o(1)},$$
which is stronger than~\eqref{eq:fst:general:bounds} when $t \ge 2s$.  However, despite all of this very significant progress, the best-known upper and lower bounds on $f_{s,t}(n)$ still differ by a polynomial factor of $n$ for every pair $(s,t)$ with $t \ge s + 2 \ge 5$. 


\pagebreak

\subsection{The Erd\H{o}s--Rogers function for pairs of graphs}

The following generalisation of the Erd\H{o}s--Rogers function was introduced recently by Mubayi and Verstra\"{e}te~\cite{MV24b}, and subsequently studied in~\cite{BCL25} and~\cite{GJS25}. 
For graphs $F$ and $H$, define $f_{H,F}(n)$ to be the largest $H$-free set of vertices that can be found in every $F$-free graph with $n$ vertices. Note that this function is only interesting when $H$ is $F$-free, and that for every $s \ge 2$, we have
\begin{equation}\label{eq:HKs:lower}
f_{K_s,K_{s+1}}(n) = f_{s,s+1}(n) \quad \qquad \text{and} \qquad \quad f_{H,K_{s+1}}(n) \ge f_s(n)
\end{equation}
for any graph $H$ that contains a copy of $K_s$. In particular, it follows from~\eqref{eq:R3k:bounds} that
\begin{equation}\label{eq:HK3:lower}
f_{H,K_3}(n) \ge f_2(n) = \Theta\big( \sqrt{n \log n} \big) 
\end{equation}
for every graph $H$ with at least one edge. Mubayi and Verstra\"{e}te~\cite{MV24b} proved that
\begin{equation}\label{eq:HK3:upper}
f_{H,K_3}(n) \le \sqrt{n} \hspace{0.03cm} \exp\big( (\log n)^{1/2+o(1)} \big)
\end{equation}
for every triangle-free graph $H$, by placing a random blow-up of $H$ in each edge of a certain linear triangle-free hypergraph, which was introduced by Erd\H{o}s, Frankl and R\"odl~\cite{EFR86} using the classical construction of Behrend~\cite{Be46} of a set without 3-term arithmetic progressions. 

The gap between~\eqref{eq:HK3:lower} and~\eqref{eq:HK3:upper} motivated Mubayi and Verstra\"{e}te~\cite[Problem~1]{MV24b} to ask if there exists a triangle-free $H$ for which $f_{H,K_3}(n) \gg \displaystyle{\sqrt{ n \log n}}$. Our method can easily be adapted to show that there are no such graphs $H$, in the following more general sense.

\begin{theorem}\label{thm:triangle-free}
Let $s \ge 3$. For every $K_s$-free graph $H$, we have
$$f_{H,K_s}(n) = O\big( \sqrt{n\log n} \, \big).$$
\end{theorem}

More precisely, for each $K_s$-free graph $H$ we will show the following (see Theorem~\ref{thm:triangle-free:quant}): for some constant 
$C_H > 0$ and all sufficiently large $n \in \N$, there exists a $K_s$-free graph with $n$~vertices in which every set of at least $C_H \displaystyle{\sqrt{ n \log n}}$ vertices contains a copy of $H$. This will turn out to follow almost immediately from the proof of Theorem~\ref{thm:ER}.

Combining Theorem~\ref{thm:triangle-free} with~\eqref{eq:JMRS:lower} and~\eqref{eq:HKs:lower}, it follows that
$$f_{H,K_{s}}(n) = \Theta\big( \sqrt{n\log n} \big)$$
for every $K_{s}$-free graph $H$ that contains a copy of $K_{s-1}$.  For $K_{s-1}$-free graphs $H$, on the other hand, Gishboliner, Janzer and Sudakov~\cite{GJS25} proved that
$$f_{H,K_{s}}(n) \le n^{1/2 \m c(H)}$$
for some $c(H) > 0$, answering another question of Mubayi and Verstra\"{e}te~\cite[Problem~2]{MV24b}. They moreover showed that
$$f_{H,K_4}(n) \ge n^{1/2 \m c(t)}$$
for some function $c(t) = \Theta(1/\sqrt{t})$ and every graph $H$ with minimum degree $t$, and that 
\begin{equation}\label{eq:GJS:bipartite}
f_{H,K_s}(n) = n^{\Theta(1/ \log s)}
\end{equation}
for all bipartite graphs $H$ with sufficiently large minimum degree. Note that~\eqref{eq:GJS:bipartite} stands in stark contrast to the case $H = K_2$, since Brada\v{c}~\cite{Br26} showed that $f_{2,s}(n) = n^{1/(s-1) \p o(1)}$.

The rest of the paper is organised as follows. First, in Section~\ref{sec:construction} we will define the random graph $G$ that we use to prove our main theorems. Next, in Section~\ref{sec:three:lemmas}, we will state two lemmas about $G$, and show that they together imply Theorem~\ref{thm:ER}. In Sections~\ref{sec:prob:bound} and~\ref{sec:open:edges} we will prove these two lemmas, and hence complete the proof of Theorem~\ref{thm:ER}. Finally, in Section~\ref{sec:triangle-free}, we will show how to adapt the construction in order to prove Theorem~\ref{thm:triangle-free}.

\section{The construction}\label{sec:construction}

In this section we will describe the probabilistic construction that we will use to prove Theorems~\ref{thm:ER} and~\ref{thm:triangle-free}. As discussed in the introduction, it will be formed by taking the (random) union of two graphs, each of which is the blow-up of a randomly-chosen collection of complete $s$-partite graphs, 
and then deleting a carefully-chosen collection of edges. It will be fairly easy to see that the graph that we construct is $K_{s+1}$-free, but significantly harder to show that every set of $C(s) \displaystyle{\sqrt{n \log n}}$ vertices contains a copy of $K_s$.

\subsection{The construction} 

Define a random graph $R(n,m,J)$ on vertex set $[n]$ by taking the union of $m$ randomly chosen copies of a graph $J$. 

\begin{definition}
The random graph $R = R(n,m,J)$ has $n$ vertices, and edge set 
$$E(R) = E(J_1) \cup \cdots \cup E(J_m),$$
where each $J_i$ is chosen independently and uniformly at random from all copies of $J$ in $K_n$. 
\end{definition}

The first step in our construction is to let $J$ be the Turán graph $T_s(\ell/r)$ (i.e., a balanced blow-up of $K_s$ with $\ell/r$ vertices) and $A$ and $B$ be independent copies of $R(n/r,m,J)$, where 
\begin{equation}\label{def:ell:m:r}
\ell = \frac{1}{2^{37} s} \sqrt{ \frac{n}{\log n} }, \qquad m = 2^{44} s^3 \sqrt{n} \gap \pigl( \log n \pig)^{3/2} \qquad \text{ and } \qquad r = n^{1/8}.
\end{equation}
Next, let $A^*$ and $B^*$ be $r$-blow-ups of $A$ and $B$, meaning that each vertex is replaced by an independent set of size $r$, and each edge is replaced by the complete bipartite graph $K_{r,r}$. We now define $G^*$ to be the random graph with vertex set $[n]$ and edge set 
$$E(G^*) = E(A^*) \cup E(B^*),$$
where the vertices of $A^*$ and $B^*$ are each mapped randomly (and independently) onto $[n]$. It will be convenient to refer to the edges of $A^*$ and $B^*$ as having different `colours'. Finally, we remove edges from $G^*$ in two steps, in order to destroy all of the copies of $K_{s+1}$:\smallskip

\noindent \emph{First deletion step}. We first remove all edges of $G^*$ that are contained in the vertex sets of two or more blow-ups of $J$ (which can be of either the same or different colours).\smallskip

\noindent \emph{Second deletion step.} For each triangle $T \subset G^*$ that is not contained in any blow-up of $J$, and which survives the first deletion step, we delete one of the edges of $T$. If the triangle contains edges from both $A^*$ and $B^*$, then the deleted edge should be of the minority\footnote{Note that all edges that are in both $A^*$ and $B^*$ are removed in the first deletion step, so if $T$ contains edges from both $A^*$ and $B^*$, then it must have one edge in $A^*$ and two in $B^*$, or vice-versa.} colour.\smallskip 

\pagebreak

We will show, for each $s \ge 2$, that with high probability the graph $G$ that results from these deletions is $K_{s+1}$-free, and moreover every $k$-subset of the vertices of size 
$$k \ge C(s)\sqrt{n\log n}$$
contains a copy of $K_s$. Before doing so, however, let us give a slightly more formal definition of the graph $G$, and introduce some notation that will be useful in the proof below. 

Recall that $A$ and $B$ are independent copies of $R(n/r,m,J)$, and let
\begin{equation}\label{def:JA:JB}
\cJ_A = \big\{ J_1^A,\ldots, J_m^A \big\} \qquad \text{ and } \qquad \cJ_B = \big\{ J_1^B,\ldots, J_m^B \big\}
\end{equation}
be the copies of $J = T_s(\ell/r)$ that form $A$ and $B$. Recall that each $J \in \cJ_A$ is an independent uniformly random copy of $T_s(\ell/r)$ with vertices in $V(A)$, and similarly for $J \in \cJ_B$, except with vertices in $V(B)$. Let the vertex sets of $A$ and $B$ be
$$V(A) = \{ u_1,\ldots,u_{n/r} \} \qquad \text{and} \qquad V(B) = \{ w_1,\ldots,w_{n/r} \}$$
and 
let $U_i$ and $W_i$ be the $r$-sets in $V(G)$ corresponding to the vertices $u_i$ and $w_i$, respectively. In particular, observe that
\begin{equation}\label{def:partitions}
V(G) = U_1 \cup \cdots \cup U_{n/r} = W_1 \cup \cdots \cup W_{n/r}
\end{equation}
are independent random partitions of $V(G)$ into sets of size $r$, and 
$$\cJ^*_A = \big\{ J^* : J \in \cJ_A \big\} \qquad \text{ and } \qquad \cJ^*_B = \big\{ J^* : J \in \cJ_B \big\}$$
are the sets of blow-ups of $J$ that form $A^*$ and $B^*$, where $J^*$ is obtained from $J$ by replacing each vertex $u_i$ by the set $U_i$ (if $J \in \cJ_A$), and $w_i$ by the set $W_i$ (if $J \in \cJ_B$). Thus
\begin{equation}\label{def:Gstar}
E(G^*) = \bigcup_{J^* \in \hspace{0.03cm} \cJ^*} E(J^*), \qquad \text{where} \qquad \cJ^* = \cJ^*_A \cup \cJ^*_B.
\end{equation}
The deletion steps are now as follows:\smallskip

\noindent \emph{First deletion step}. Define
\begin{equation}\label{def:D1}
D_1 = \big\{ \hspace{0.02cm} xy \in E(G^*) : \{x,y\} \subset V(J_1^*) \cap V(J_2^*) \text{ for some } J_1^*,J_2^* \in \cJ^* \text{ with } J_1^* \ne J_2^* \hspace{0.02cm} \big\},
\end{equation}
and let $G'$ be the graph with vertex set $V(G^*)$ and edge set $E(G^*) \setminus D_1$.\smallskip 

\noindent \emph{Second deletion step.} Let $\cT$ be the collection of triangles $T \subset G'$ such that $V(T) \not\subset V(J^*)$ for every $J^* \in \cJ^*$. For each $T \in \cT$, choose an edge of $T$, following the rule that if $T$ has exactly one edge of a given colour ($A$ or $B$), then that edge is always chosen. 

Let $D_2$ be the collection of chosen edges, and let $G$ be the graph with vertex set $V(G')$ and edge set $E(G') \setminus D_2$.\smallskip 

Before continuing, let us make a simple but important observation.

\pagebreak

\begin{observation}\label{obs:unique:Jstar}
For any edge $xy \in E(G)$, there is a unique $J^* \in \cJ^*$ with $x,y \in V(J^*)$.
\end{observation}

\begin{proof}
This follows immediately from the definitions. Indeed, since $xy \in E(G^*)$, it follows from~\eqref{def:Gstar} that there must exist a $J^* \in \cJ^*$ with $x,y \in V(J^*)$. Moreover, since $xy \not\in D_1$, it follows from~\eqref{def:D1} that there cannot exist two such members of $\cJ^*$.
\end{proof}

We can now show the first key property of the graph $G$. 

\begin{lemma}\label{lem:G-Ks+1-free} 
$G$ is $K_{s+1}$-free.
\end{lemma}

\begin{proof}
Let $S \subset V(G)$ be a set of $s+1$ vertices, and suppose that $G[S] = K_{s+1}$. Suppose first that $S \subset V(J^*)$ for some $J^* \in \cJ^*$. By Observation~\ref{obs:unique:Jstar}, 
it follows that $G[S] \subset J^*[S]$. But $J^*$ is $s$-partite, so this is a contradiction. 

On the other hand, if $S \not\subset V(J^*)$ for every $J^* \in \cJ^*$, then there must exist three vertices $x,y,z \in S$ that are not all contained in the same member of $\cJ^*$. Indeed, for any pair of vertices $x,y \in S$, let $J^*$ be the unique element of $\cJ^*$ with $x,y \in V(J^*)$ given by Observation~\ref{obs:unique:Jstar}, and let $z \in S \setminus V(J^*)$. But then one of the edges of the triangle $xyz$ would have been removed in the second deletion step, so we again have a contradiction. 
\end{proof}

\section{Two key lemmas that together imply Theorem~\ref{thm:ER}}\label{sec:three:lemmas}

In this section we will state our two key lemmas, and show that, together with Lemma~\ref{lem:G-Ks+1-free}, they imply the following quantitative version of the upper bound in Theorem~\ref{thm:ER}. 

\begin{theorem}\label{thm:ER:again}
There exists $C > 0$ such that the following holds for all $s \ge 2$ and $n \ge s^C$. There exists a $K_{s+1}$-free graph $G$ with $n$ vertices such that if
$$k \ge C s^3 \displaystyle{\sqrt{n\log n}},$$
then every set of $k$ vertices of $\,G$ contains a copy of $K_s$. 
\end{theorem}

Since the conclusion of the theorem is monotone in $k$, it will suffice to fix
\begin{equation}\label{def:k}
k = 2^{40} s^3 \sqrt{n \log n}
\end{equation}
and show that with high probability\footnote{Note that, by choosing $C$ sufficiently large, we may assume that $n$ is sufficiently large, and therefore that all events that hold `with high probability' do indeed hold with probability arbitrarily close to $1$.} every set of $4k$ vertices of the random graph $G$ that we constructed in Section~\ref{sec:construction} contains a copy of $K_s$. Let us remark here that with extra work the dependence of $k$ on $s$ in Theorem~\ref{thm:ER:again} can be improved to $O\pig( s^{3/2} \displaystyle{\sqrt{\log s}} \hspace{0.04cm} \pig)$ (see Section~\ref{discussion:sec}), but for simplicity we have chosen to prove the slightly weaker bound stated here. 

The first step is to observe that our two random partitions of $V(G)$ are likely to have the following useful property. Define the \emph{projection} of a set $S \subset V(G)$ onto $V(A)$ to be 
$$\pi_A(S) = \big\{ u_i \in V(A) : U_i \cap S \ne \emptyset \big\}$$ 
and define the projection $\pi_B$ onto $V(B)$ similarly. Note that the minimum of $|\pi_A(S)|$ and $|\pi_B(S)|$ can be as small as $|S|/r$ (this will be important later). 
If $|S| \ge 4k$, however, then with high probability the projection of $S$ onto \emph{one} of the partitions has size at least~$k$. 

\begin{observation}\label{lem:partition-property}
With high probability, we have 
\begin{equation} 
\max\big\{ |\pi_A(S)|,  |\pi_B(S)| \big\} \ge k
\end{equation}
for every set $S \subset V(G)$ with $|S| \ge 4k$.
\end{observation}

\begin{proof}
Fix a set $S \subset V(G)$ with $|S| \ge 4k$, and consider the probability that $|\pi_A(S)| < k$. By symmetry, this is the same as fixing the partition and choosing the set $S$ uniformly at random. Thus
$$\Pr\big( |\pi_A(S)| < k \big) \le \binom{n/r}{k}\binom{rk}{4k} \binom{n}{4k}^{-1},$$
since we can first choose $k$ of the $n/r$ vertices of $A$, and then choose the elements of $S$ from the union of the corresponding parts. Since the partitions are independent, and observing that $r^3 k \le s^3 n^{7/8} \log n \ll n$ for $n \ge s^C$, by~\eqref{def:ell:m:r} and~\eqref{def:k},  
it follows that
$$\Pr\Big( \max\big\{ |\pi_A(S)|,  |\pi_B(S)|\big\} < k \Big) \le \bigg(\frac{r^3 n}{k} \bigg)^{2k} \binom{n}{4k}^{-2} \le \bigg(\frac{2^4 r^3 k}{n} \bigg)^{2k} \binom{n}{4k}^{-1} \ll \binom{n}{4k}^{-1},$$
since $\pig( \frac{a}{b} \pig)^b \leq {a \choose b} \leq \pig( \frac{ea}{b} \pig)^b$. The claimed result now follows by a union bound over $S$.
\end{proof}

We will observe later (see Lemma~\ref{lem:S-hit-I}) that if $|\pi_A(S)| \ge k$, then (with sufficiently high probability to allow a union bound over $S$) at least half of the graphs $J^* \in \cJ^*_A$ add at least one copy of $K_s$ to $S$. Our main task will be to bound the probability that all of these copies of $K_s$ are destroyed during the two edge-deletion steps. 

To do so, we will need the following useful notion. Given a graph $H$ and a set of vertices $S \subset V(H)$, we define the set of \emph{closed} edges in $S$ with respect to the graph $H$ to be
\begin{equation}\label{def:closed}
X_H(S) = E\pig( H[S] \pig) \cup E\pig( H^2[S] \pig),
\end{equation}
where $H^2$ is the square of $H$. That is, the edge $xy$ is included in $X_H(S)$ if $x,y \in S$ and there is a path of length at most two in $H$ between $x$ and $y$. To motivate this definition, let us make the following important observation. 

\begin{observation}
Suppose that the edge $e \in E(A^*)$ is removed in one of the two deletion steps, and choose $J^* \in \cJ^*_A$ with $e \in E(J^*)$. Then $e$ must either be closed with respect to $B^*$, or closed with respect to the union of the graphs in $\cJ^*_A \setminus \{ J^* \}$.
\end{observation}

\begin{proof}
Recall that if $e = xy$ is removed in the first deletion step, then $\{x,y\}$ is contained in the vertex sets of two different members of $\cJ^*$. Now, observe that if $J_1^* \in \cJ^*$, then every pair of vertices in $V(J_1^*)$ is closed with respect to $J_1^*$, and therefore in this case $e$ is closed with respect to one of the graphs in $\cJ^* \setminus \{ J^* \}$.

Similarly, if $e$ is removed in the second deletion step, then there exists a path $xzy$ of length two between $x$ and $y$ in either $A^*$ or $B^*$, with  $z \not\in V(J^*)$. It again follows that $e$ is closed with respect to either $B^*$ or the union of two members of $\cJ^*_A \setminus \{ J^* \}$, as claimed.
\end{proof}

Set $\beta = 2^{-7} s^{-2}$, and define, for each set $S \subset V(G)$, the event 
$$\cE_\beta(S) = \Big\{ \pigl| X_{A^*}(S) \cup X_{B^*}(S) \hspace{0.02cm} \pig| \le \beta k^2 \Big\}$$
that there are not too many closed edges in $S$ with respect to either $A^*$ or $B^*$. We are now ready to state the two key lemmas that together imply Theorem~\ref{thm:ER:again}. The first of these, which we will prove in Section~\ref{sec:prob:bound}, says that if $\cE_\beta(S)$ holds for a set $S$ with projection at least $k$ onto one of the partitions, then $G[S]$ is extremely likely to contain a copy of $K_s$. 

\begin{lemma}\label{lem:ind-set-calc}
Let $\hspace{0.02cm} S \subset V(G)$. If $\hspace{0.03cm} |\pi_A(S)| \ge k$, then 
$$\Pr\Big( \cE_\beta(S) \cap \big\{ G[S] \text{ is $K_s$-free} \big\} \Big) \le e^{-m/16}.$$
\end{lemma}

The proof of Lemma~\ref{lem:ind-set-calc} is not too difficult: the idea is to reveal the elements of $\cJ_A$ in a random order, and note that each removed edge was closed by members of $\cJ^*$ that were revealed earlier in the order with probability at least $1/3$. Since the event $\cE_\beta(S)$ implies that a random copy of $K_s$ contains a closed edge with probability at most $\beta s^2$, the claimed bound follows via a simple calculation.

Our second key lemma bounds the probability that the event $\cE_\beta(S)$ fails to hold.

\begin{lemma}\label{lem:main-open-edges} 
With high probability, we have 
$$\pig| \hspace{0.015cm} X_{A^*}(S) \pig| \le \beta \binom{k}{2},$$
for every $S \subset V(G)$ with $|S| = k$. 
\end{lemma}

We will prove Lemma~\ref{lem:main-open-edges} in Section~\ref{sec:open:edges}. The basic idea is that the expected density of closed edges in the graph $A$ is small because it only has $n/r$ vertices, and the density of closed edges is not significantly affected by blowing up the graph. The main challenge is to deal with subsets of $V(A)$ that have large intersection with many members of~$\cJ_A$, which requires some careful counting, but is otherwise again reasonably straightforward. 

Let us now show how to deduce our main theorem from the two lemmas above. 

\begin{proof}[Proof of Theorem~\ref{thm:ER:again}, assuming Lemmas~\ref{lem:ind-set-calc} and~\ref{lem:main-open-edges}]
Let $G$ be the graph constructed in Section~\ref{sec:construction}, and recall from Lemma~\ref{lem:G-Ks+1-free} that $G$ is $K_{s+1}$-free. It therefore remains to prove that, with high probability, every set of size $4k$ contains a copy of $K_s$. 

First, by Observation~\ref{lem:partition-property}, with high probability the partitions~\eqref{def:partitions} of $V(G)$ used to blow up the graphs $A$ and $B$ are chosen so that
$$\max\big\{ |\pi_A(I)|,  |\pi_B(I)|\big\} \ge k$$
for every set $I \subset V(G)$ with $|I| \ge 4k$. Fix such a set $I$, and choose a subset $S \subset I$ with $|S| = k$ and $|\pi_A(S)| = |S|$, say. Observe that, by Lemma~\ref{lem:ind-set-calc}, we have
$$\Pr\Big( \cE_\beta(S) \cap \big\{ G[S] \text{ is $K_s$-free} \big\} \Big) \le e^{-m/16},$$
and that with high probability the event $\cE_\beta(S)$ holds for every $k$-set $S$, by Lemma~\ref{lem:main-open-edges}. 

Taking a union bound over the $\binom{n}{k}$ choices for the set $S$, it follows that 
$$\Pr\Big( \exists \, I \subset V(G) : |I| \ge 4k \text{ and } G[I] \text{ is $K_s$-free} \Big) \le \binom{n}{k} e^{-m/16} + o(1) = o(1)$$
as required, since $m = 16 k \log n$, by~\eqref{def:ell:m:r} and~\eqref{def:k}.
\end{proof}

\section{Sets with few closed edges are unlikely to be $K_s$-free}\label{sec:prob:bound}

The main aim of this section is to prove Lemma~\ref{lem:ind-set-calc}, which bounds the probability that a set with projection $k$ is $K_s$-free. Recall that we are given a set $I \subset V(G)$ with $|\pi_A(I)| \ge k$, and are required to show that 
$$\Pr\Big( \cE_\beta(I) \cap \big\{ G[I] \text{ is $K_s$-free} \big\} \Big) \le e^{-m/16}.$$
Note that the events $\cE_\beta(I)$ and $K_s \not\subset G[I]$ are both decreasing properties of the set $I$, so (replacing $I$ by a subset if needed) we may assume that $|I| = |\pi_A(I)| = k$. 

To begin, for each $k$-set $T \subset V(A)$ we need a lower bound on the size of the set
$$Q_A(T) = \big\{ i \in [m] :  K_s \subset J^A_i\pig[ T \pig] \big\}.$$
That is, we need to bound the number of $J \in \cJ_A$ which contain a copy of $K_s$ in the set $T$. The following lemma crucially uses the fact that $k \ell / n \ge s \log s$; its failure for smaller values of $k$ is a key bottleneck for the dependence on $s$ required for our construction. 

\begin{lemma}\label{lem:S-hit-I} 
For every $T \subset V(A)$ with $|T| \ge k$ we have 
$$\Pr\bigg( |Q_A(T)| \le \frac{m}{2} \bigg) \le e^{-m}.$$
\end{lemma}

\begin{proof} 
Observe that $|Q_A(T)|$ is a binomial random variable, and that 
$$\Pr\big( K_s \not\subset J\pig[ T \pig] \big) \le s^{-7}$$ 
for each $J \in \cJ_A$. Indeed, each of the $s$ parts of $J$ is disjoint from the set $T$ with probability at most
$$\bigg( 1 - \frac{kr}{n} \bigg)^{\ell/sr} \le \exp\bigg( - \frac{k \ell}{sn} \bigg) \le s^{-8},$$
since $k \ell / n = 8 s^2 \ge 8 s \log s$. The claimed bound now follows by Chernoff's inequality.
\end{proof}



We are now ready to prove the main lemma of the section. 

\begin{proof}[Proof of Lemma~\ref{lem:ind-set-calc}] 
Let $T = \pi_A(I)$. By Lemma~\ref{lem:S-hit-I}, it will suffice to show that 
\begin{equation}\label{eq:ind-set-calc:aim}
\Pr\Big( \big\{ |Q_A(T)| \ge m/2 \big\} \cap \cE_\beta(I) \cap \big\{ G[I] \text{ is $K_s$-free} \big\} \Big) \le e^{-m/15}.
\end{equation}
To do so, we reveal the entire graph $B$, and then reveal the graphs $J \in \cJ_A$ in a uniformly random order. To be precise, let $\pi$ be a uniformly random permutation of $[m]$, and for each $0 \le i \le m$ and each set $S \subset V(G)$, define $X_\pi^{(i)}(S)$ to be the set of closed edges in $S$ after revealing $B$ and  $J^A_{\pi(1)},\ldots,J^A_{\pi(i)}$. That is
$$X_\pi^{(i)}(S) = X_{B^*}(S) \cup X_{A_\pi^*(i)}(S),$$
where $A_\pi^*(i)$ is the union of the blow-ups of the graphs $J^A_{\pi(1)},\ldots,J^A_{\pi(i)}$. 

Now, for each $i \in Q_A(T)$, let $J_i^*$ be the blow-up of $J_i^A$, and choose a set $S_i \subset I$ such that $J_i^*[S_i] = K_s$, uniformly at random from the available choices. Now define
$$Y_\pi = \sum_{i \ina Q_A(T)} \1\big[ X_\pi^{(i-1)}(S_{\pi(i)}) \ne \emptyset \big],$$
so $Y_\pi$ is the number of sets $\{ S_i : i \in Q_A(T) \}$ that contain a closed edge at the moment they are exposed, using the order $\pi$. We make the following key observation. 

\begin{claim}\label{permutation:claim}
If $\hspace{0.03cm} G[I]$ is $K_s$-free, then 
$$\Ex_\pi\pig[ Y_\pi \pig] \ge \frac{|Q_A(T)|}{3}.$$
\end{claim}

\begin{proof}[Proof of claim] 
Since $G[I]$ is $K_s$-free, for each $i \in Q_A(T)$ the clique with vertex set $S_i$ must have an edge $e$ deleted from it during one of the two deletion steps. If $e$ was deleted in the first deletion step, then it must also be contained in some other $J^* \in \cJ^*$, and this $J^*$ was revealed before $J^A_i$ with probability at least $1/2$. Moreover, if it was revealed first, then $e$ was already closed at step $\pi(i)$, since every pair of vertices of each $J^* \in \cJ^*$ is closed.\footnote{Note that in the case $s = 2$ this is true because we included the edges of $H[I]$ in the set $X_H(I)$.} 

Similarly, if $e$ was deleted in the second deletion step, then it was closed by a pair of elements of $\cJ^*$, and (recalling that $B$ was revealed first) with probability at least $1/3$, both of these graphs were revealed before $J^A_i$. It follows that each element of $Q_A(T)$ contributes~$1$ to the sum in the definition of $Y_\pi$ with probability at least $1/3$, as required.
\end{proof}

Since $Y_\pi \le |Q_A(T)|$, it follows that if $G[I]$ is $K_s$-free, then
$$\Pr_\pi\left( Y_\pi \ge \frac{|Q_A(T)|}{6} \right) \ge \frac{1}{6},$$
and hence the left-hand side of~\eqref{eq:ind-set-calc:aim} is at most
\begin{equation}\label{eq:ind-set-calc:next:aim}
6 \cdot \Ex_\pi\Big[ \Pr\Big( \cE_\beta(I) \cap \big\{ Y_\pi \ge m/12 \big\} \Big) \Big].
\end{equation}

To bound the probability inside the expectation, we will partition the event according to the set $Q_A(T)$, and reveal the elements of $\cJ_A$ in the order $\pi$, conditioning on the events $\{ S_i \subset I \}$ for each $i \in Q_A(T)$, and on the event that $J_i^A\pig[ T \pig]$ is $K_s$-free for each $i \not\in Q_A(T)$. Note that all copies of $K_s$ in $J^*$ are equivalent, so by symmetry we can fix any of them to be the pre-image of $S_i$. Moreover, since $|\pi_A(I)| = |I|$, under this conditioning $S_i$ is a uniform subset of $I$ of size $s$, and therefore if $\cE_\beta(I)$ holds, then 
\begin{equation}\label{eq:supersaturation}
\Pr\Big( X_\pi^{(i-1)}(S_{\pi(i)}) \ne \emptyset \;\big|\; \cF_{i-1}^\pi \Big) \le \beta k^2 {s \choose 2} {k \choose 2}^{-1} \le \beta s^2 = 2^{-7}
\end{equation}
for each $i \in [m]$ such that $\pi(i) \in Q_A(T)$, where $\cF_{i-1}^\pi$ is the $\sigma$-algebra generated by the random graphs $B$ and $J^A_{\pi(1)},\ldots,J^A_{\pi(i-1)}$. It follows that
$$\Pr\Big( \cE_\beta(I) \cap \big\{ Y_\pi \ge m/12 \big\} \Big) \le {m \choose m/12} \big( \beta s^2 \big)^{m/12} \le \big( 12e \cdot 2^{-7} \big)^{m/12} \le e^{-m/12},$$
for every permutation $\pi$ and every set $Q_A(T)$. Since~\eqref{eq:ind-set-calc:next:aim} is an upper bound on the left-hand side of~\eqref{eq:ind-set-calc:aim}, it follows that
$$\Pr\Big( \big\{ |Q_A(T)| \ge m/2 \big\} \cap \cE_\beta(I) \cap \big\{ G[I] \text{ is $K_s$-free} \big\} \Big) \le 6 \cdot e^{-m/12},$$
and since $m \ge s^{C/2}$, this implies that~\eqref{eq:ind-set-calc:aim} holds, as required. 
\end{proof}

\section{Every $k$-set contains few closed edges}\label{sec:open:edges}

In this section we will prove Lemma~\ref{lem:main-open-edges}, which bounds the number of closed edges with respect to $A^*$ in a $k$-set. The lemma turns out to be a straightforward consequence of the following similar statement in the graph $A$. Recall that $\beta = 2^{-7} s^{-2}$, and let us say that an event holds with \emph{very high probability} if it fails with probability at most $n^{-1}$.

\begin{lemma}\label{lem:A-open-edges} 
Let $k/2r \le t = O(s^3 k/r)$. With very high probability, we have 
$$\pig| \hspace{0.015cm} X_A(T) \pig| \le \frac{\beta}{8} {t \choose 2}$$
for every set $T \subset V(A)$ with $|T| = t$.
\end{lemma}

In order to prove Lemma~\ref{lem:A-open-edges}, we will first prove several simple pseudorandom properties of the family $\cJ_A$, and then show that the claimed bound on $|X_A(T)|$ follows deterministically from these properties. Our first property of $\cJ_A$ is a bound on the maximum degree of the $\ell/r$-uniform hypergraph $\cA$ with vertex set $V(A)$ and edge set $\big\{ V(J) : J \in \cJ_A \big\}$. 

\begin{lemma}\label{obs:deg}
With very high probability, $\cA$ has maximum degree at most $2^{10} s^2 \log n$.
\end{lemma}

\begin{proof}
Let $d_\cA(u) = |\{ J \in \cJ_A : u \in V(J) \}|$ be the degree of a vertex $u$ in $\cA$. We have 
$$\Pr\big( d_\cA(u) \ge a \big) \le \binom{m}{a} \bigg(\frac{\ell}{n}\bigg)^{a} \le \bigg(\frac{2^{7} e s^2 \log n}{a} \bigg)^a \le \frac{1}{n^2}$$
if $a \ge 2^{10} s^2 \log n$, since $m \ell = 2^7 s^2 n \log n$, by~\eqref{def:ell:m:r}.  
Taking a union bound over vertices $u \in V(A)$ now gives the claimed bound. 
\end{proof}

Our second property is that small sub-hypergraphs of $\cA$ have few high-degree vertices.  

\begin{lemma}\label{lem:few:vtxs:many:nbrs}
With very high probability, 
$$\big|\big\{ v \in V(A) : d_\cH(v) \ge 8 \big\}\big| \le e(\cH)$$
for every $\cH \subset \cA$ with $e(\cH) \le n^{1/4}$. 
\end{lemma} 

\begin{proof}
The probability that there exists $\cH \subset \cA$ with $d \ge e(\cH)$ vertices of degree at least $8$ and at most $n^{1/4}$ edges is at most
$${m \choose e(\cH)} {n/r \choose d} {e(\cH) \choose 8}^d \bigg( \frac{\ell}{n} \bigg)^{8d} \le \bigg( m^{1/8} \cdot n^{1/8} \cdot e(\cH) \cdot \frac{\ell}{n} \bigg)^{8d} \le n^{-d/2},$$
where the first step holds since $e(\cH) \le d$, and in the second we used~\eqref{def:ell:m:r} and the bound $e(\cH) \le n^{1/4}$. Taking a union bound over the choices of $d$ and $e(\cH)$, the lemma follows. 
\end{proof}

Now, define $H$ to be the random graph with vertex set $[m]$ 
and edge set
$$ij \in E(H) \qquad \Leftrightarrow \qquad V(J_i^A) \cap V(J_j^A) \ne \emptyset.$$
To motivate this definition, observe that if
$$\pig| V(J^A_i) \cap T \pig| = a_i$$
for each $i \in [m]$, then
\begin{equation}\label{eq:XAT:basic:bound} 
\pig| \hspace{0.015cm} X_A(T) \pig| \le \hspace{0.02cm} \sum_{i \eqa 1}^m {a_i \choose 2} + \sum_{ij \ina E(H)} a_i a_j.
\end{equation}

Our next two properties say that $H$ is sparse on large sets. Given 
sets $Y,Z \subset V(H)$, let us write $e_H(Y,Z)$ for the number of edges of $H$ with one endpoint in $Y$ and the other in $Z$. We will find it convenient in the proofs below to set $C = 2^{10}$ and $c = 2^{-6}$. 

\begin{lemma}\label{lem:intersection:sparse}
With very high probability, we have
\begin{equation}\label{eq:intersection}
e_H(Y,Z) \le n^{-c} \gap |Y| |Z| + C^2 s^2 \big( |Y| + |Z| \big) \log n 
\end{equation}
for all sets $Y,Z \subset V(H)$ with $|Y| + |Z| = O\pig( s^2 t \log n \pig)$. 
\end{lemma}

\begin{proof}
Let $\{v_1,\ldots,v_d\}$ be the set of vertices that are contained in $V(J_y^A) \cap V(J_z^A)$ for some $y \in Y$ and $z \in Z$ with $y \ne z$. For each $i \in [d]$, let
\begin{equation}\label{def:ai:bi}
\big|\big\{ y \in Y  : v_i \in V(J_y^A) \big\} \big| = a_i \qquad \text{and} \qquad \big|\big\{ z \in Z : v_i \in V(J_z^A) \big\} \big| = b_i,
\end{equation}
and suppose without loss of generality that $\sum_i a_i \le \sum_i b_i$. Note that
$$e_H(Y,Z) \le \sum_{i \eqa 1}^d a_i b_i \le \Delta(\cA) \sum_{i \eqa 1}^d a_i,$$
and recall from Lemma~\ref{obs:deg} that we may assume that $\Delta(\cA) \le C s^2 \log n$. It follows that if~\eqref{eq:intersection} fails to hold for the pair $(Y,Z)$, then
$$|Y| + |Z| \le c \sum_{i \eqa 1}^d a_i \qquad \text{and} \qquad |Y| |Z| \le n^{2c} \sum_{i \eqa 1}^d a_i.$$
Now, there are at most ${m \choose |Y|} {m \choose |Z|}$ choices for $Y$ and $Z$, and ${n/r \choose d}$ choices for the vertices $v_i$, given $d$. Given the sequences $\va$ and $\vb$, there are then at most $\prod_i {|Y| \choose a_i} {|Z| \choose b_i}$ choices for the sets that contain $v_1,\ldots,v_d$, and each event $v_i \in J$ occurs with probability\footnote{More precisely, this holds conditional on any subset of the other events occurring.} at most $\ell/n$. 

The expected number of bad pairs $(Y,Z)$ for a given choice of $(|Y|,|Z|,d,\va,\vb)$ is thus at most
\begin{equation}\label{eq:intersection:count}
{m \choose |Y|} {m \choose |Z|} {n/r \choose d} \prod_{i = 1}^d {|Y| \choose a_i} {|Z| \choose b_i} \bigg( \frac{\ell}{n} \bigg)^{a_i + b_i}. 
\end{equation}
To bound~\eqref{eq:intersection:count}, observe first that if $d \le c \sum_i b_i$, then~\eqref{eq:intersection:count} is at most
$$\bigg( m^c \cdot |Y| \cdot \frac{\ell}{n} \bigg)^{\sum_i a_i} \bigg( m^c \cdot n^c \cdot |Z| \cdot \frac{\ell}{n} \bigg)^{\sum_i b_i} \le n^{-c \sum_i (a_i + b_i)} \le n^{-d},$$
since $|Y| + |Z| \le c \sum_i a_i \le c \sum_i b_i$ and $|Y| + |Z| = O\pig( s^2 t \log n \pig) \le n^{2/5}$, since $t = O(s^3k/r)$. 

On the other hand, if $d \ge c \sum_i b_i$, then we have
$${n/r \choose d} \le \bigg( \frac{e n^{1+3c}}{r|Y|^2} \bigg)^{d/2} \bigg( \frac{e n^{1+3c}}{r|Z|^2} \bigg)^{d/2} \le \bigg( \frac{n^{1/2+2c}}{\sqrt{r} |Y|} \bigg)^{\sum_i a_i} \bigg( \frac{n^{1/2+2c}}{\sqrt{r} |Z|} \bigg)^{\sum_i b_i},$$
where in the first step we used $d \ge c \sum_i b_i \ge n^{-3c} \gap |Y||Z|$, and in the second we used the bounds $d \le \sum_i a_i \le \sum_i b_i$ and $r \pig( |Y|^2 + |Z|^2 \pig) \le n$. It follows that~\eqref{eq:intersection:count} is at most
$$m^{|Y|+|Z|} \bigg( \frac{n^{1/2+2c}}{\sqrt{r}} \cdot \frac{\ell}{n} \bigg)^{\sum_i (a_i + b_i)} \le n^{-c \sum_i (a_i + b_i)} \le n^{-cd},$$
since $|Y| + |Z| \le c \sum_{i \eqa 1}^d a_i \le c \sum_i b_i$, and recalling that $\ell \le \sqrt{n}$, $r = n^{1/8}$ and $c = 2^{-6}$. 

Finally, we take a union bound over the choices of $(|Y|,|Z|,d,\va,\vb)$. Noting that there are at most $\Delta(\cA)^{2d}$ choices for the sequences $\va$ and $\vb$, and recalling from Lemma~\ref{obs:deg} that $\Delta(\cA) = O( s^2 \log n)$ with very high probability,  the lemma follows.
\end{proof}

When $Y$ and $Z$ have very different sizes, we will need the following stronger bound. 

\begin{lemma}\label{lem:intersection:asym}
With very high probability, we have
\begin{equation}\label{eq:intersection:asym}
e_H(Y,Z) \le C |Z|
\end{equation}
for every set $Y,Z \subset V(H)$ with $|Y| \le n^c$ and $s^4 |Y| (\log n)^3 \le |Z| \le O\pig( s^2 t \log n \pig)$. 
\end{lemma}

\begin{proof}
Define the set $\{v_1,\ldots,v_d\}$ and the sequences $\va$ and $\vb$ as in~\eqref{def:ai:bi}, and observe that, by~Lemmas~\ref{obs:deg} and~\ref{lem:few:vtxs:many:nbrs}, with high probability we have $\max\{ a_i, b_i \} \le \Delta(\cA) \le C s^2 \log n $, and that at most $|Y|$ of the $a_i$ satisfy $a_i \ge 8$. It follows that, with high probability, we have
$$C |Z| \le e_H(Y,Z) \le \sum_{i \eqa 1}^d a_i b_i \le 7 \sum_{i \eqa 1}^d b_i  + O\pig( s^4 |Y| (\log n)^2 \pig) \le 7 \sum_{i \eqa 1}^d b_i + |Z|$$
for every pair $(Y,Z)$ with $|Z| \ge s^4 |Y| (\log n)^3$ such that~\eqref{eq:intersection:asym} fails to hold, and also
$$s^4 |Y| (\log n)^3 \le |Z| \le e_H(Y,Z) \le \sum_{i \eqa 1}^d a_ib_i \le C s^2 \log n \sum_{i \eqa 1}^d a_i.$$
In particular, for each such pair $(Y,Z)$, we have 
$$|Y| \le c \sum_{i \eqa 1}^d a_i \qquad \text{and} \qquad |Z| \le c \sum_{i \eqa 1}^d b_i.$$ 
Moreover, recall from~\eqref{eq:intersection:count} that the expected number of bad pairs $(Y,Z)$ for a given choice of $(|Y|,|Z|,d,\va,\vb)$ is at most
\begin{equation}\label{eq:intersection:asym:count}
{m \choose |Y|} {m \choose |Z|} {n/r \choose d} \prod_{i = 1}^d {|Y| \choose a_i} {|Z| \choose b_i} \bigg( \frac{\ell}{n} \bigg)^{a_i + b_i}. 
\end{equation}
To bound~\eqref{eq:intersection:asym:count}, suppose first that $d \le c \sum_i b_i$, and observe that
$${m \choose |Y|} \prod_{i = 1}^d {|Y| \choose a_i} \bigg( \frac{\ell}{n} \bigg)^{a_i} \le \bigg( n^c \cdot |Y|  \cdot \frac{\ell}{n} \bigg)^{\sum_i a_i} \le n^{-c \sum_i a_i},$$
since $|Y| \le c \sum_i a_i$ and $|Y| \le n^c$, and that
$${m \choose |Z|} {n/r \choose d} \prod_{i = 1}^d {|Z| \choose b_i} \bigg( \frac{\ell}{n} \bigg)^{b_i} \le \bigg( n^{2c} \cdot |Z| \cdot \frac{\ell}{n} \bigg)^{\sum_i b_i} \le n^{-c \sum_i b_i},$$
since $|Z| + d \le 2c \sum_i b_i$ and $|Z| = O\pig( s^2 t \log n \pig) = O\pig( s^5 (k/r) \log n \pig) \le n^{2/5}$. 

On the other hand, if $d \ge c \sum_i b_i$, then we have
$${n/r \choose d} \le \bigg( \frac{en}{r} \bigg)^{d/2} \bigg( \frac{en}{r|Z|^2} \bigg)^{d/2} \le \bigg( \sqrt{ \frac{en}{r} } \bigg)^{\sum_i a_i} \bigg( \sqrt{\frac{en}{r|Z|^2}} \bigg)^{\sum_i b_i},$$
where the first step follows from $d \ge c \sum_i b_i \ge |Z|$, and in the second we used the bounds $d \le \sum_i a_i$, $d \le \sum_i b_i$, and $r |Z|^2 = O\pig( s^8 (k^2/r) (\log n)^2 \pig) \le n$. It follows that~\eqref{eq:intersection:asym:count} is at most
$$m^{|Y| + |Z|} \bigg( \sqrt{ \frac{en}{r} } \cdot |Y| \cdot \frac{\ell}{n} \bigg)^{\sum_i a_i} \bigg( \sqrt{ \frac{en}{r|Z|^2} } \cdot |Z| \cdot \frac{\ell}{n} \bigg)^{\sum_i b_i} \le n^{-c \sum_i (a_i + b_i)}$$
since $|Y| \le c \sum_i a_i$, $|Z| \le c \sum_i b_i$, $r = n^{1/8}$ and $|Y| \le n^c$. 

Finally, we again take a union bound over the choices of $(|Y|,|Z|,d,\va,\vb)$. Since there are at most $\Delta(\cA)^{2d}$ choices for the sequences $\va$ and $\vb$, and $\Delta(\cA) = O( s^2 \log n)$ with high probability, the expected number of bad pairs $(Y,Z)$ is $o(1)$, as required.  
\end{proof}

The final property we need is that the sum of large intersections with $T$ isn't too large.

\begin{lemma}\label{lem:few:big:intersections}
With very high probability, 
\begin{equation}\label{eq:few:big:intersections}
\sum_{J \ina \cJ_A} |V(J) \cap T| \cdot \1\big[ |V(J) \cap T| \ge C \big] \le Ct
\end{equation}
for every set $T \subset V(A)$ with $|T| = t$.
\end{lemma}

\begin{proof}
We first observe the following bound on the upper tail of $|V(J) \cap T|$:
$$\Pr\big( |V(J) \cap T| \ge a \big) \le {\ell/r \choose a} \bigg( \frac{r |T|}{n} \bigg)^a \le r^{-a/2} = n^{-a/16},$$
since $|T| = O\pig( s^2 k / r \pig)$ and $k\ell = O\pig( s^2 n \pig)$. Next, observe that if~\eqref{eq:few:big:intersections} fails to hold, then for some $t$-set $T \subset V(A)$, there exists a set $\{ J_1, \ldots, J_d \} \subset \cJ_A$ and a sequence $\va = (a_1,\ldots,a_d)$, with $\sum_i a_i = Ct$ and $a_i \ge C$ for each $i \in [d]$, such that $|V(J_i) \cap T| \ge a_i$ for each $i \in [d]$. Since these events are independent, a union bound gives 
$${n/r \choose t} \sum_{d = 1}^t {m \choose d} {Ct \choose d} n^{-Ct/16} \le n^{-t},$$
as claimed.
\end{proof}

We can now deduce our bound on $X_A(T)$. 

\begin{proof}[Proof of Lemma~\ref{lem:A-open-edges}]
Given a $t$-set $T \subset V(A)$, define 
\begin{equation}\label{def:ai:D} 
\pig| V(J^A_i) \cap T \pig| = a_i \qquad \text{and} \qquad D = \big\{ i \in [m] : a_i \ge C \big\},
\end{equation}
where $\cJ_A = \big\{ J^A_1, \ldots, J^A_m \big\}$, and recall from~\eqref{eq:XAT:basic:bound} that 
\begin{equation}\label{eq:XAT:basic:bound:again} 
\pig| \hspace{0.015cm} X_A(T) \pig| \le \hspace{0.03cm} \sum_{i \eqa 1}^m {a_i \choose 2} + \sum_{ij \ina E(H)} a_i a_j.
\end{equation}
Now, by Lemmas~\ref{obs:deg} and~\ref{lem:few:big:intersections}, with very high probability we have
\begin{equation}\label{eq:few:big:intersections:app}
\sum_{i \ina D} a_i \le Ct \qquad \text{and} \qquad \sum_{i \eqa 1}^m a_i \le \Delta(\cA) \cdot t \le C s^2 t \log n.
\end{equation}
Moreover, by Lemmas~\ref{lem:intersection:sparse} and~\ref{lem:intersection:asym}, with very high probability we have 
\begin{equation}\label{eq:intersection:again}
e_H(Y,Z) \le n^{-c} \gap |Y| |Z| + C^2 s^2 \big( |Y| + |Z| \big) \log n
\end{equation}
for every $Y,Z \subset V(H)$ with $|Y| + |Z| = O\pig( s^2 t \log n \pig)$, and  
\begin{equation}\label{eq:intersection:asym:again}
e_H(Y,Z) \le C |Z|
\end{equation}
whenever $Y$ and $Z$ additionally satisfy $|Y| \le n^c$ and $|Z| \ge s^4 |Y| (\log n)^3$. 

We will show that if these four properties are satisfied by $\cJ_A$ and $T$, then the right-hand side of~\eqref{eq:XAT:basic:bound:again} is deterministically at most $(\beta/8) {t \choose 2}$. Note first that the first sum 
is easily bounded: since $|J| = \ell/r$ for every $J \in \cJ_A$, it follows from~\eqref{def:ai:D} and~\eqref{eq:few:big:intersections:app} that 
\begin{equation}\label{eq:XAT:bound1} 
\sum_{i \eqa 1}^m {a_i \choose 2} \le \gap \frac{\ell}{r} \gap \sum_{i \ina D} a_i + C \gap \sum_{i \eqa 1}^m a_i \le \frac{C \ell t}{r} + C^2 s^2 t \log n \le \frac{t^2}{n^c},
\end{equation}
since $t \ge k/2r$ and $k \ge 2^{12} \ell \log n \ge n^{1/4}$, by~\eqref{def:ell:m:r} and~\eqref{def:k}. To bound the second term, define
$$
D_0 = \big\{ i \in [m] : 1 \le a_i < C \big\} \qquad \text{and} \qquad D_x = \big\{ i \in D : 2^{x - 1} < a_i \le 2^x \big\}
$$
for each $x \in \N$, let $L = \pig\lceil \hspace{-0.02cm} \log_2(\ell/r) \hspace{0.01cm} \pig\rceil$, and observe that 
\begin{equation}\label{eq:final:thing:to:bound}
\sum_{ij \ina E(H)} a_i a_j \le C^2 \sum_{x \eqa 0}^L \sum_{y \eqa 0}^L 2^{x + y} e_H(D_x, D_y),
\end{equation}
since if $x > L$ then $D_x = \emptyset$. To bound the right-hand side of~\eqref{eq:final:thing:to:bound}, observe first that 
\begin{equation}\label{eq:product:bound}
\sum_{x \eqa 0}^L \sum_{y \eqa 0}^L 2^{x + y} |D_x| \gap |D_y| = \bigg( \sum_{x \eqa 0}^L 2^x |D_x| \bigg)^2 \le \bigg( 2 \sum_{i \eqa 1}^m a_i \bigg)^2 = O\big( s^4 t^2 (\log n)^2 \big),
\end{equation}
since $\sum_i a_i = O\pig( s^2 t \log n \pig)$, and that 
$$\sum_{x \eqa 1}^L \sum_{y \eqa 1}^L 2^{x + y} \big( |D_x| + |D_y| \big) \le \bigg( \sum_{x \eqa 1}^L 2^x |D_x| \bigg) \bigg( \sum_{y \eqa 1}^L 2^{y + 1} \bigg) \le \frac{C^2 t \ell}{r},$$
since $\sum_{x \gea 1} 2^x |D_x| \le 2 \sum_{i \ina D} a_i \le 2Ct$, by~\eqref{eq:few:big:intersections:app}, and $2^L \le 2\ell/r$. By~\eqref{eq:intersection:again}, it follows that
\begin{equation}\label{eq:bounding:xy:not:zero}
\sum_{x \eqa 1}^L \sum_{y \eqa 1}^L 2^{x + y} e_H(D_x, D_y) \le \frac{C^4 s^2 t \ell \log n}{r} + \frac{t^2}{n^{c/2}}.
\end{equation}
To deal with the terms involving $D_0$, observe first that\footnote{Here the set $D_0$ might be too small to apply~\eqref{eq:intersection:asym:again} directly, but since we only require the upper bound~\eqref{eq:intersection:asym:app}, in this case we can simply replace it by any superset of size at most $C s^2 t \log n$.}   
\begin{equation}\label{eq:intersection:asym:app}
e_H(D_0,D_x) \le C^2 s^2 t \log n
\end{equation}
if $|D_x| \le n^c$, by~\eqref{eq:few:big:intersections:app} and~\eqref{eq:intersection:asym:again}, since $|D_0| \le \sum_i a_i \le C s^2 t \log n$, and hence
$$\sum_{|D_x| \lea n^c} 2^{x} e_H(D_0, D_x) \le \sum_{x \eqa 1}^L 2^x C^2 s^2 t \log n \le \frac{C^3 s^2 t \ell \log n}{r},$$
since $2^L \le 2\ell/r$. For larger sets $D_x$, on the other hand, we claim that
$$
\sum_{|D_x| \gea n^c} 2^x \big( |D_0| + |D_x| \big) \le \frac{t \cdot |D_0|}{n^{c/3}} + \sum_{x \eqa 0}^L 2^x |D_x| \le \frac{t^2}{n^{c/4}},
$$
since $\sum_{x \gea 0} 2^x |D_x| \le 2 \sum_i a_i = O\pig( s^2 t \log n \pig)$, and hence $|D_x| \gea n^c$ implies that $2^x \le n^{-c/2} \cdot t$. Combining this with~\eqref{eq:intersection:again} and~\eqref{eq:product:bound}, it follows that
\begin{equation}\label{eq:bounding:y:zero}
\sum_{x \eqa 0}^L 2^{x} e_H(D_0, D_x) \le \frac{C^3 s^2 t \ell \log n}{r} + \frac{t^2}{n^{c/5}},
\end{equation}
and hence, putting the pieces together, by~\eqref{eq:XAT:basic:bound:again},~\eqref{eq:XAT:bound1},~\eqref{eq:final:thing:to:bound},~\eqref{eq:bounding:xy:not:zero} and~\eqref{eq:bounding:y:zero}, we obtain
\begin{equation}\label{eq:XAT:final:bound}
\pig| \hspace{0.015cm} X_A(T) \pig| \le \frac{C^6 s^2 t \ell \log n}{r} + \frac{t^2}{n^{c/6}} \le \frac{\beta}{8} {t \choose 2},
\end{equation}
as required, since $t \ge k/2r$ and $k = 2^{77} s^4 \ell \log n$. 
\end{proof}

\pagebreak

We are finally ready to prove Lemma~\ref{lem:main-open-edges}. To do so, we simply apply Lemma~\ref{lem:A-open-edges} to a random subset of the vertices of $A$, weighted by the size of the intersections $U_i \cap S$.

\begin{proof}[Proof of Lemma~\ref{lem:main-open-edges}]
Let $S \subset V(G)$ with $|S| = k$, and set $X_A = X_A(V(A))$. Observe that
$$\pig| \hspace{0.015cm} X_{A^*}(S) \pig| \le \hspace{0.02cm} rk + \hspace{0.02cm} \sum_{i,j} \pig| U_i \cap S \pig|  \cdot \pig| U_j \cap S \pig| \cdot \1\big[ u_i u_j \in X_A \big],$$
since at most $rk$ pairs of vertices in $S$ are contained in some part $U_i$, and if $x \in U_i \cap S$ and $y \in U_j \cap S$ with $i \ne j$, then $xy \in X_{A^*}(S)$ if and only if $u_i u_j \in X_A$. 

Now, define a random set $T \subset \pi_A(S)$ by including each $u_i \in \pi_A(S)$ independently with probability $|U_i \cap S| / r$. Note $\Ex\pig[  |T| \pig] = k/r$ and that 
$$\Ex\big[ | X_A(T) | \big] = \sum_{i, \gap j} \frac{\pig| U_i \cap S \pig|}{r} \cdot \frac{\pig| U_j \cap S \pig|}{r} \cdot \1\big[ u_i u_j \in X_A \big].$$
Since $k/2r \le |T| \le 2k/r$ with very high probability, by Chernoff, and $| X_A(T) | \le |T|^2$ deterministically, it follows from Lemma~\ref{lem:A-open-edges} that, with high probability,
$$\pig| \hspace{0.015cm} X_{A^*}(S) \pig| \le \hspace{0.02cm} rk + r^2 \cdot \Ex\big[ | X_A(T) | \big] \le \frac{\beta r^2}{4} {2k/r \choose 2} \le \beta \binom{k}{2}$$
for every $k$-set $S \subset V(G)$, as required. 
\end{proof}

\section{$K_s$-free graphs with no large $H$-free subsets}\label{sec:triangle-free}

To begin this final section, let us observe that the proof of Theorem~\ref{thm:ER} can easily be adapted to prove the following theorem, which immediately implies Theorem~\ref{thm:triangle-free}. 

\begin{theorem}\label{thm:triangle-free:quant}
Let $s \ge 3$. For every $K_s$-free graph $H$, there exists $C_H > 0$ such that the following holds for all sufficiently large $n \in \N$. There exists a $K_s$-free graph $G$ with $n$ vertices in which every set of at least $C_H \displaystyle{\sqrt{n \log n}}$ vertices contains a copy of $H$. 
\end{theorem}

\begin{proof}
Define a graph $G$ as in Section~\ref{sec:construction}, but with $J$ equal to a balanced blow-up of $H$ with $\ell/r$ vertices. To show that $G$ is $K_s$-free, we repeat the proof of Lemma~\ref{lem:G-Ks+1-free}: if $G[S] = K_s$ then, since $J$ is $K_s$-free, there must exist three vertices $x,y,z \in S$ that are not all contained in $V(J^*)$ for any $J^* \in \cJ^*$. But then one of the edges of the triangle $xyz$ would have been removed in the second deletion step.

Now, the proof of Theorem~\ref{thm:ER} (applied with $s = v(H)$) implies that if $k \ge C_H \displaystyle{ \sqrt{n \log n} }$, then with high probability every set of $k$ vertices of $G$ contains one vertex from each part of some blow-up $J^* \in \cJ^*$ such that none of the edges between these vertices were removed in the deletion steps. These vertices therefore induce a copy of $H$ in $G$, as required. 
\end{proof}

Note that, since we proved Theorem~\ref{thm:ER:again} with a constant of order $s^3$, the argument above proves Theorem~\ref{thm:triangle-free:quant} with $C_H = O(v(H)^3)$. In the remainder of this final section we will discuss what we expect to be the correct dependence on $s$ for our construction.  

\subsection{The dependence of $f_s(n)$ on $s$}\label{discussion:sec}

With a little extra effort, the constant $C(s) = O(s^3)$ given by Theorem~\ref{thm:ER:again} can be improved to 
\begin{equation}\label{eq:optimal:bound}
C(s) = O\big( \hspace{0.02cm} s^{3/2} \sqrt{\log s} \hspace{0.04cm} \big).
\end{equation}
To see this, observe first that the inequalities that we needed $k$, $\ell$ and $m$ to satisfy during the proof were
\begin{equation}\label{eq:klm:needs}
\frac{k \ell}{n} \ge 8 s \log s, \qquad m \ge 16 k \log n \qquad \text{and} \qquad k \ge \frac{C s^2 m \ell^2}{n}
\end{equation}
for some constant $C > 0$, the first in the proof of Lemma~\ref{lem:S-hit-I}, the second in the final step of the proof of Theorem~\ref{thm:ER:again}, and the third in~\eqref{eq:XAT:final:bound}, the final step of the proof of Lemma~\ref{lem:A-open-edges}, where we have replaced a factor of $s^2 \log n$ by $m\ell / n$, which is the bound on $\Delta(\cA)$ given by Lemma~\ref{obs:deg}. These bounds together imply that 
$$k \ge C s^2 \log s \, \sqrt{n \log n},$$
so to obtain~\eqref{eq:optimal:bound} we need to work a little harder. 

To obtain the claimed bound, it suffices to show that the weaker inequalities 
\begin{equation}\label{eq:klm:better}
\frac{k \ell}{n} \ge C s \log s, \qquad m \log s \ge C k \log n \qquad \text{and} \qquad k \ge \frac{C s m \ell^2}{n}
\end{equation}
suffice for the proof. To do so, we need to deal with `unstructured' sets, for which we can assume that the projection is much larger than $k/r$, separately from `structured' sets, for which we have a much smaller union bound. 
The proof in the unstructured case is the same as above, but we can now take $\beta = s^{-3}$, say, which then allows us to obtain a bound of the form $s^{-\Theta(m)}$ in Lemma~\ref{lem:ind-set-calc}. In the structured case we can only prove Lemma~\ref{lem:A-open-edges} with $\beta = \Theta(s^{-1})$, due to the weaker bound on $k$, and therefore need to modify the proof of Lemma~\ref{lem:ind-set-calc}, using a standard supersaturation theorem for copies of $K_s$ in~\eqref{eq:supersaturation}. This gives a bound of the form $\exp\pig( - e^{-O(\beta s^2)} m \pig)$, which suffices in this case (as long as $n$ is sufficiently large as a function of $s$), since we only need to union bound over $s^{O(k)}$ sets. 


We believe that the bound~\eqref{eq:optimal:bound} is best possible for our construction, and therefore represents a natural barrier to further progress. To see this, observe that the first inequality in~\eqref{eq:klm:better} is needed so that most $k$-sets contain a copy of $K_s$ in a given copy of $T_s(\ell)$, 
the second is necessary for the union bound over $k$-sets, and the third is needed because the union of $s - 1$ neighbourhoods in the graph $A^*$ does not contain any open copy of $K_s$, and therefore any copy of $K_s$ in $G^*$ in this set will be destroyed in the deletion steps. 

Finally, let us note one further natural barrier, which does not depend on our particular choice of deletion steps. Since every neighbourhood in a $K_{s+1}$-free graph is $K_s$-free, we cannot hope to improve the third inequality in~\eqref{eq:klm:better} by more than a factor of $s$, and together with the first two inequalities this implies that 
$$k = \Theta\big( s \sqrt{\log s} \, \sqrt{n \log n} \big)$$
is a natural limit for any construction similar to ours.

\section*{Acknowledgements}

The work that led to this paper began during the PCMI summer program in July 2025, and continued during a visit of JS to IMPA in February 2026. The authors are grateful to both institutions for providing a wonderful working environment. They would also like to thank Huy Pham for pointing them to the matching lower bound in~\cite{JMRS21}.

\end{document}